\documentclass[10pt]{article}
\font\smallit=cmti10

\usepackage{amssymb,latexsym,amsmath,epsfig,amsthm} 
\usepackage{url}

\makeatletter

\renewcommand\section{\@startsection {section}{1}{\z@}
{-30pt \@plus -1ex \@minus -.2ex}
{2.3ex \@plus.2ex}
{\normalfont\normalsize\bfseries}}

\renewcommand\subsection{\@startsection{subsection}{2}{\z@}
{-3.25ex\@plus -1ex \@minus -.2ex}
{1.5ex \@plus .2ex}
{\normalfont\normalsize\bfseries}}

\renewcommand{\@seccntformat}[1]{\csname the#1\endcsname. }

\makeatother

\newtheorem{theorem}{Theorem}
\newtheorem{corollary}[theorem]{Corollary}

\newenvironment{definition}[1][Definition]{\begin{trivlist}
\item[\hskip \labelsep {\bfseries #1}]}{\end{trivlist}}
\newenvironment{example}[1][Example]{\begin{trivlist}
\item[\hskip \labelsep {\bfseries #1}]}{\end{trivlist}}
\newenvironment{remark}[1][Remark]{\begin{trivlist}
\item[\hskip \labelsep {\bfseries #1}]}{\end{trivlist}}

\begin{document}

\begin{center}
\uppercase{\bf On Permutiples Having a Fixed Set of Digits}
\vskip 20pt
{\bf Benjamin V. Holt}\\
{\smallit Department of Mathematics, Humboldt State University, Arcata, CA 95521, USA}\\
{\tt bvh6@humboldt.edu}\\
\vskip 10pt
\end{center}
\vskip 30pt

\vskip 30pt

\centerline{\bf Abstract}

\noindent
A permutiple is the product of a digit preserving multiplication, that is, a number
which is an integer multiple of some permutation of its digits.
Certain permutiple problems, particularly transposable, cyclic, and, more recently, palintiple numbers,
have been well-studied. In this paper we study the problem of general digit preserving multiplication. 
We show how the digits and carries of a permutiple are related
and utilize these relationships to develop methods for finding new permutiple examples from old.
In particular, we shall focus on the problem of finding new permutiples 
from a known example having the same set of digits.

\pagestyle{myheadings}
\thispagestyle{empty}
\baselineskip=12.875pt
\vskip 30pt 

\section{Introduction}

A \textit{permutiple} is a natural number with the property of being an integer multiple of some permutation of its digits. 
Digit permutation problems are nothing new \cite{guttman, sutcliffe} and have been a topic of 
study for both amateurs and professionals alike \cite{kalman}.
A relatively well-studied example of permutiples includes \emph{palintiple} numbers, 
also known as reverse multiples \cite{kendrick_1,sloane,young_1}, which are 
integer multiples of their digit reversals and include well-known base-10 examples such as
$87912=4\cdot 21978$ and $98901 = 9 \cdot 10989$.
As noted by Sutcliffe \cite{sutcliffe} in his seminal paper on palintiple numbers,
cyclic digit permutations such as $714285 = 5 \cdot 142857$ are also well-studied examples. 
We also note that 142857 is an example of a \textit{cyclic number}; not only does multiplication by 5 permute the digits, 
but 2,3,4, and 6 also produce cyclic digit permutations.

Permutiples for which the digits are cyclically permuted 
are relatively well-understood and their description is fairly straightforward in comparison to palintiples.
The digits of cyclic permutiples are found in repeating base-$b$ decimal expansions of the form $a/p$ where 
$a<p$ and $p$ is a prime which does not divide $b$ \cite{guttman, kalman}.  
On the other hand, palintiples (digit reversing permutations)
admit quite a variety of classifications \cite{kendrick_1, kendrick_2} and are not nearly as well-understood.
Young \cite{young_1, young_2}, building upon the body of work of Sutcliffe \cite{sutcliffe} and others \cite{grimm, kolsmol}, 
translates the palintiple problem into graph-theoretical language by representing 
an efficient palintiple search method as a tree-graph where the possible carries are represented as nodes 
and the potential digits are associated with the edges.
Continuing the work of Young \cite{young_1, young_2}, Sloane \cite{sloane} modified Young's tree-graph representation
into the \emph{Young graph} which is a visualization of digit-carry palintiple structure.
The paper identifies and studies several Young graph isomorphism classes which describe palintiple type. 
Furthering the work of Sloane \cite{sloane}, Kendrick \cite{kendrick_1} proves two of Sloane's \cite{sloane}
main conjectures involving Young graph isomorphism classes which describe two well-understood palintiple types. 
The work of Holt \cite{holt_1,holt_3} takes a more elementary approach and classifies palintiples
according to patterns exhibited by their carries. This approach, as noted by Kendrick \cite{kendrick_1}, seems to coincide with 
certain Young graph isomorphism classes with Holt \cite{holt_3} 
conjecturing that the class of palintiples characterized by the 1089 graph are 
precisely the collection of \textit{symmetric palintiples} 
(the carry sequence is palindromic) described in \cite{holt_1}.
Kendrick's work \cite{kendrick_1, kendrick_2} reveals the sheer multitude of palintiple types 
when classified according to Young graph isomorphism.

In this paper we establish some general properties of digit preserving multiplication. 
We generalize the results for palintiple numbers found in 
Young \cite{young_1}, Sloane \cite{sloane}, Kendrick \cite{kendrick_1}, and Holt \cite{holt_1} 
to an arbitrary permutation.
Using these results, we develop methods for finding new permutiples from old. 
In particular, we consider the problem of finding new base-$b$ permutiples with multiplier $n$ 
having a fixed set of digits from a single known example. 
Moreover, we find a condition under which our methods give us all permutiples of a particular base and multiplier
having the same digits as a known example. 

\section{Permutiple Digits and Carries}

We begin with a definition. We shall use $(d_k,d_{k-1},\ldots,d_0)_b$ to denote the 
natural number $\sum_{j=0}^{k}d_j b^j$ where each $0\leq d_j<b$.
\begin{definition}
Let $n$ be a natural number and $\sigma$ be a permutation on $\{0,1,2,\ldots, k\}$.
We say that $(d_k, d_{k-1},\ldots, d_0)_b$  is an $(n,b,\sigma)$-\textit{permutiple} provided 
\[
(d_k,d_{k-1},\ldots,d_1, d_0)_b=n(d_{\sigma(k)},d_{\sigma(k-1)},\ldots, d_{\sigma(1)}, d_{\sigma(0)})_b.
\]
\end{definition}

Using the language established above, letting 
$\rho=\left( 
\begin{array}{ccccc}
0 & 1 & 2 & 3 & 4 \\
4 & 3 & 2 & 1 & 0 \\
\end{array}
\right)$,
87912 is a $(4,10,\rho)$-permutiple since $87912=4\cdot 21978$.

In order to avoid introducing extra digits when multiplying, it is assumed that $n<b$. 
We note, however, that in order to circumvent overly cumbersome theorem statements, we do allow for leading zeros. 
Letting $\varepsilon$ be the identity permutation, every natural number is a $(1,b,\varepsilon)$-permutiple. 
Such trivial examples are ignored so that $n>1$.  Furthermore, $b=2$ implies that
$n=1$. Therefore, we impose the additional restriction that $b \neq 2$. Thus, hereafter, we assume that
$n$ and $b$ are natural numbers such that $1<n<b$.  

The following two theorems are an exceedingly straightforward generalization of the first and third theorems
of Holt \cite{holt_1} which concern palintiple numbers.
A description of single-digit multiplication in general is as follows: 
let $p_j$ denote the $j$th digit of the product, $c_j$ the $j$th
carry, and $q_{j}$ the $j$th digit of the number being multiplied by $n$. Then the
iterative algorithm for single-digit multiplication is 
\[
\begin{split}
c_0&=0\\
p_j&=\lambda(n q_{j}+c_j)\\
c_{j+1}&=\left[ nq_{j}+c_j-\lambda(nq_{j}+c_j) \right]\div b
\end{split}
\]
where $\lambda$ gives the least non-negative residue modulo $b$. Since
$(p_k, p_{k-1}, \ldots, p_0)_b$ is a $(k+1)$-digit number, $c_{k+1}=0$.
For any $(n,b,\sigma)$-permutiple, $(d_k, d_{k-1},\ldots, d_0)_b$, 
 $q_j=d_{\sigma(j)}$ so that $d_{j}=p_j=\lambda(nd_{\sigma(j)}+c_j)$. Hence, we have our first result.

\begin{theorem}
Let $(d_k, d_{k-1},\ldots, d_0)_b$ be an $(n,b,\sigma)$-permutiple and let $c_j$ be the
$j$th carry. Then
\[b c_{j+1}-c_j=nd_{\sigma(j)}-d_{j}\]
for $0\leq j \leq k$.
\label{fund}
\end{theorem}

As is the case for palintiples, the following shows that the carries of any permutiple
are less than the multiplier.

\begin{theorem}
Let $(d_k, d_{k-1},\ldots, d_0)_b$ be an $(n,b,\sigma)$-permutiple and let $c_j$ be the
$j$th
carry. Then $c_j\leq n-1$ for all $0 \leq j \leq k$.
\label{carries}
\end{theorem}
\begin{proof}
The proof will proceed by induction. We have $c_0=0\leq n-1$. Now suppose that $c_j\leq n-1$.
For a contradiction suppose $c_{j+1} \geq n$. Then Theorem \ref{fund} implies
$bc_{j+1}-c_j+d_{j}=nd_{\sigma(j)}$. By our inductive hypothesis we have
$bn-(n-1)=(b-1)n+1 \leq nd_{\sigma(j)}$. Therefore, $d_{\sigma(j)}>b-1$ which is a
contradiction.
\end{proof}

The following is a converse to Theorem \ref{fund}.

\begin{theorem}
Suppose $b c_{j+1}-c_j=nd_{\sigma(j)}-d_{j}$ for all $0\leq j \leq k$ where
$(d_k,\ldots, d_0)$ is a $(k+1)$-tuple of base-$b$ digits and $(c_k,\ldots,c_0)$ 
is a $(k+1)$-tuple of base-$n$ digits such that $c_0=0$.
Then $(d_k,\ldots, d_0)_b$ is an $(n,b,\sigma)$-permutiple with carries $c_k,\ldots, c_0$.
\label{fund_conv}
\end{theorem}

\begin{proof}
By a simple calculation,
\[
\sum_{j=0}^{k} (nd_{\sigma(j)}-d_{j})b^j=\sum_{j=0}^{k}(b c_{j+1}-c_j)b^j=0
\]
so that $(d_k,\ldots, d_0)_b$ is an $(n,b,\sigma)$-permutiple. 
Letting $(\hat{c}_k,\ldots,\hat{c}_0)$ be the carries, an 
application of Theorem \ref{fund} and a simple induction argument establish 
that $\hat{c}_j=c_j$ for all $0\leq j \leq k$. 
\end{proof}

Letting $\psi$ be the $(k+1)$-cycle $(0,1,2,\ldots,k)$, 
it is convenient to write the relations between the digits and the carries
found in Theorems \ref{fund} and \ref{fund_conv} in matrix form,
\begin{equation}
(b P_\psi-I)\textbf{c} = (nP_\sigma-I) \textbf{d},
\label{matrix}
\end{equation}
where $I$ is the identity matrix, $P_\psi$ and $P_\sigma$ are permutation matrices,
and $\textbf{c}$ and $\textbf{d}$ are column vectors containing the carries and digits, respectively.
We note that these matrices are indexed from 0 to $k$ rather than from 1.

We highlight that we will extensively use the fact   
that the number 
$(d_k,\ldots,d_0)_b$ is an $(n,b,\sigma)$-permutiple with carries
$c_k$, $\ldots$ , $c_1$, $c_0=0$ if and only if Equation \ref{matrix} holds 
(a consequence of Theorems \ref{fund} and \ref{fund_conv}).

Multiplying both sides of Equation \ref{matrix} by
$\sum_{\ell=0}^{|\sigma|-1}(nP_\sigma)^{\ell}$, we can express the digits in terms of the carries as
$\textbf{d}=\frac{1}{n^{|\sigma|}-1}\sum_{\ell=0}^{m-1}(n^{\ell} P_{\sigma^{\ell}})(b P_\psi-I)\textbf{c}.$
Similarly, multiplying both sides by
$\sum_{\ell=0}^{k}(bP_\psi)^{\ell}$, we can likewise express the carries in terms of the digits.
$\textbf{c}=\frac{1}{b^{k+1}-1}\sum_{\ell=0}^{k}(b^{\ell} P_{\psi^{\ell}})(n P_\sigma-I)\textbf{d}.$
In component form, we have the following result.
\begin{theorem} 
Let $(d_k, d_{k-1},\ldots, d_0)_b$ be an $(n,b,\sigma)$-permutiple, and let $c_j$ be the
$j$th
carry. Then
\[
d_j=\frac{1}{n^{|\sigma|}-1}\sum_{\ell=0}^{|\sigma|-1} (bc_{\psi \sigma^{\ell}(j)}-c_{\sigma^{\ell}(j)})n^{\ell}
\]
and
\[
c_j=\frac{1}{b^{k+1}-1}\sum_{\ell=0}^{k} (nd_{\sigma \psi^{\ell}(j)}-d_{\psi^{\ell}(j)})b^{\ell}
\]
for all $0\leq j \leq k$.
\label{digits_and_carries}
\end{theorem}

\begin{remark}
 We direct the reader's attention to the symmetry between the above equations expressing the digits in terms of the 
 carries and vice versa.
 The above also generalizes the relationship between palintiple numbers and their carries found 
 in Young \cite{young_1}, Sloane \cite{sloane}, and Holt \cite{holt_1,holt_3}.
\end{remark}

The next theorem places restrictions on the first digit of any permutiple.

\begin{theorem}
For any nontrivial $(n, b,\sigma)$-permutiple, $(d_k ,d_{k-1},\ldots,d_0)_b$, $\gcd(n,b)$ divides $d_0$.
\label{order_sigma}
\end{theorem}

\begin{proof}
Let $(d_k , d_{k-1},\ldots, d_0)_b$ be an $(n, b,\sigma)$-permutiple where $c_j$ is the $j$th carry. 
Then, for $j = 0$, Theorem \ref{digits_and_carries} gives 
$
(n^{|\sigma|}-1)d_0
=b \left(\sum_{\ell=0}^{|\sigma|-1} c_{\psi \sigma^{\ell}(0)}n^{\ell}\right)-n\left(\sum_{\ell=1}^{|\sigma|-1} c_{\sigma^{\ell}(0)}n^{\ell-1}\right).
$
Thus, $\gcd(n,b)$ divides $d_0$ since $n$ and $n^{|\sigma|}-1$ are relatively prime.
\end{proof}

\section{New Permutiples from Old}

We shall now consider the problem of finding new permutiples from known examples.
The approach taken here, as stated in the beginning,
will be to restrict our attention to finding new permutiples having the same digits as our known example.
The ultimate aim of our effort is to answer the question of whether or not all permutiples having the same digits
can be found from a single example. If not, are there conditions under which it is possible?
The following results give us some methods for constructing new permutiples from old.

\begin{theorem}
Let $(d_k,\ldots,d_0)_b$ be an $(n,b,\sigma)$-permutiple with carries $c_k$, $\ldots$ , $c_0$, and let 
$\mu$ be a permutation such that $c_{\mu(0)}=0$. 
Then $(d_{\pi(k)}, d_{\pi(k-1)},\ldots, d_{\pi(0)})_b$ is an $(n,b,\pi^{-1}\sigma \pi)$-permutiple
with carries $c_{\mu(k)}$, $c_{\mu(k-1)}$, $\ldots$ , $c_{\mu(0)}$ if and only if
$P_{\pi}(bP_{\psi}-I)\textbf{c}=(bP_{\psi}-I)P_{\mu}\textbf{c}$.
\label{new_perm}
\end{theorem}

\begin{proof}
By our hypothesis, Equation \ref{matrix} is satisfied.  
Thus, if $(d_{\pi(k)}, d_{\pi(k-1)},\ldots, d_{\pi(0)})_b$ is an $(n,b,\pi^{-1}\sigma \pi)$-permutiple
with carries $c_{\mu(k)}$, $c_{\mu(k-1)}$, $\ldots$ , $c_{\mu(0)}$, then
\[P_{\pi}(b P_{\psi}-I)\textbf{c}
=P_{\pi}(nP_{\sigma}-I)\textbf{d}
=(nP_{\pi^{-1}\sigma \pi}-I)P_{\pi}\textbf{d}
=(b P_{\psi}-I)P_{\mu}\textbf{c}.\]
Conversely, if $P_{\pi}(bP_{\psi}-I)\textbf{c}=(bP_{\psi}-I)P_{\mu}\textbf{c}$, then
\[(nP_{\pi^{-1}\sigma \pi}-I)P_{\pi}\textbf{d}
=P_{\pi}(nP_{\sigma}-I)\textbf{d}
=P_{\pi}(b P_{\psi}-I)\textbf{c}
=(b P_{\psi}-I)P_{\mu}\textbf{c}.\]
\end{proof}

\begin{corollary}
Let $(d_k, d_{k-1},\ldots, d_0)_b$ be an $(n,b,\sigma)$-permutiple with carries $c_k$, $c_{k-1}$, $\ldots$ , $c_0$. 
If $c_j=0$, then $(d_{\psi^j(k)}, d_{\psi^{j}(k-1)},\ldots, d_{\psi^j(1)}, d_{\psi^j(0)})_b$
is an $(n,b,\psi^{-j} \sigma \psi^j)$-permutiple with carries 
$c_{\psi^j(k)}$, $c_{\psi^{j}(k-1)}$, $\ldots$ , $c_{\psi^j(1)}$, $c_{\psi^j(0)}=c_j=0$.
\label{zero_carry}
\end{corollary}

\begin{remark}
 We note that the above corollary follows either from Theorem \ref{digits_and_carries} by setting $c_j=0$,
 or from Theorem \ref{new_perm} by setting $\pi=\mu=\psi^{j}$.
\end{remark}

\begin{example}
 Consider the $(4,10,\rho)$-permutiple $(8,7,9,1,2)_{10}=4 \cdot (2,1,9,7,8)_{10}$. 
 Performing routine multiplication, we see that the carries are $(c_4,c_3,c_2,c_1,c_0)=(0,3,3,3,0)$. 
 Not surprisingly, applying Corollary \ref{zero_carry} to $j=0$ yields the original permutiple.
 However, for the case of $j=4$, we have $(7,9,1,2,8)_{10}=4 \cdot (1,9,7,8,2)_{10}$ with carries
$(c_{\psi^4(4)},c_{\psi^4(3)},c_{\psi^4(2)},c_{\psi^4(1)},c_{\psi^4(0)})
=(c_{3},c_{2},c_{1},c_{0},c_{4})
=(3,3,3,0,0)$.
\end{example}

Applying Corollary \ref{zero_carry} a bit more generally, 
if $(d_k, d_{k-1},\ldots, d_0)_b$ is any $(n,b)$-palintiple such that $c_k=0$ 
(this includes all symmetric, doubly-derived, and doubly-reverse-derived palintiples \cite{holt_3}), then
$(d_{k-1}, d_{k-2}, \ldots, d_{0}, d_k)_b$ is an $(n,b,\psi^{-k} \rho \psi^{k})$-permutiple.

Setting $\mu$ to the identity permutation in Theorem \ref{new_perm}, we obtain another useful corollary.

\begin{corollary}
 Let $(d_k, d_{k-1},\ldots, d_0)_b$ be an $(n,b,\sigma)$-permutiple with carries $c_k$, $c_{k-1}$, $\ldots$ , $c_0$.
 Then $(d_{\pi(k)}, d_{\pi(k-1)},\ldots, d_{\pi(0)})_b$ is an $(n,b,\pi^{-1}\sigma \pi)$-permutiple
with carries $c_k$, $c_{k-1}$, $\ldots$ , $c_0$ if and only if $P_{\pi}(bP_{\psi}-I)\textbf{c}=(bP_{\psi}-I)\textbf{c}.$
\label{new_perm_id}
\end{corollary}

\begin{example}
 Consider again the base-10 palintiple $(8,7,9,1,2)_{10}=4 \cdot (2,1,9,7,8)_{10}$
 with carries $(c_4,c_3,c_2,c_1,c_0)=(0,3,3,3,0)$.
With the above in mind, we calculate
\[
 (10P_{\psi}-I)\textbf{c}=
 \left(
 10 \cdot
 \left[
 \begin{matrix}
 0 & 1 & 0 & 0 & 0\\
 0 & 0 & 1 & 0 & 0\\
 0 & 0 & 0 & 1 & 0\\
 0 & 0 & 0 & 0 & 1\\
 1 & 0 & 0 & 0 & 0\\
 \end{matrix}
 \right]
 -
 \left[
 \begin{matrix}
 1 & 0 & 0 & 0 & 0\\
 0 & 1 & 0 & 0 & 0\\
 0 & 0 & 1 & 0 & 0\\
 0 & 0 & 0 & 1 & 0\\
 0 & 0 & 0 & 0 & 1\\
 \end{matrix}
 \right] 
 \right)
 \left[
 \begin{matrix}
    0\\
   3\\
   3\\
   3\\
   0\\
 \end{matrix}
\right] 
 =
 \left[
 \begin{matrix}
   30\\
   27\\
   27\\
   -3\\
    0\\
 \end{matrix}
\right].
\]

Since the above column vector is unchanged by $P_{\pi}$, where $\pi$ is the transposition $(1,2)$, 
we see by Corollary \ref{new_perm_id} that $(8,7,1,9,2)_{10}$ is a $(4,10, (1,2) \rho (1,2))$-permutiple 
with carries $(c_4,c_3,c_2,c_1,c_0)=(0,3,3,3,0)$ which may also be be confirmed by simple arithmetic.

Performing the same calculation as above, 
the $(4,10,\psi^{-4}\rho\psi^{4})$-permutiple, $(7, 9, 1, 2, 8)_{10}=4 \cdot (1, 9, 7, 8, 2)_{10}$, 
from the previous example yields via Corollary \ref{new_perm_id} the $(4,10,((2,3)\psi^{-4}\rho \psi^4 (2,3))$-permutiple 
$(7, 1, 9, 2, 8)_{10}=4 \cdot (1, 7, 9, 8, 2)_{10}$ with carries $(c_4,c_3,c_2,c_1,c_0)=(3,3,3,0,0)$.
We note that we arrive at the same result by applying Corollary \ref{zero_carry}
to the $(4,10, (1,2) \rho (1,2))$-permutiple $(8,7,1,9,2)_{10}$ which is affirmed by the fact that
$(1,2)\psi^4=\psi^4 (2,3)$. 
\end{example} 

At this point, several questions naturally present themselves. 
Can all permutiples having a particular set of digits be found by repeated use of Theorem \ref{new_perm} and its corollaries?
One does not have to look far to see that the answer is no. If we consider the example
$(7, 8, 9, 1, 2)_{10} =4  \cdot (1, 9, 7, 2, 8)_{10}$ with carries $(c_4,c_3,c_2,c_1,c_0)=(3,2,1,3,0)$,
we see that the results obtained thus far do not account for this example since the carries are different.

Another question is if we have an $(n,b,\sigma)$-permutiple, $(d_k, \ldots, d_0)_b$, with carries $c_k$, $\ldots$ , $c_0$, and
an $(n,b,\tau)$-permutiple, $(d_{\pi(k)}, \ldots, d_{\pi(0)})_b$, with permuted carries 
$c_{\mu(k)}$, $c_{\mu(k-1)}$, $\ldots$ , $c_{\mu(0)}$, must it be that $\tau = \pi^{-1} \sigma \pi$?
Again, with a little effort, we can find an example which shows that this is not always the case. 
Consider $(4, 3, 5, 1, 2)_6=2\cdot(2, 1, 5, 3, 4)_6$, a $(2,6,\sigma)$-permutiple, and 
$(2, 5, 1, 3, 4)_6 = 2 \cdot (1, 2, 3, 4, 5)_6$, a $(2,6,\tau)$-permutiple,
both with the same carry vector $(0,1,1,1,0)$.
Now $\sigma=(0,4)(1,3)$, $\pi=(0,4)(3,2,1)$, and $\tau=(0,3,4,2,1)$, but $\tau\neq \pi^{-1}\sigma\pi$.

Thus, it is clear that our results so far do not account for every possibility. 
Therefore, we shall require some additional machinery in order to find every permutiple with the same digits
as our known example. Again, for the purpose of less cumbersome theorem statements, we shall henceforth 
assume that $(d_k, \ldots, d_0)_b$ is an $(n,b,\sigma)$-permutiple.
We begin with a definition, motivated by the above, which will help us to organize and classify our new examples.

\begin{definition}
 We say that an $(n,b, \tau_1)$-permutiple, $(d_{\pi_1(k)}, d_{\pi_1(k-1)},\ldots, d_{\pi_1(0)})_b$, and
 an $(n,b, \tau_2)$-permutiple, $(d_{\pi_2(k)}, d_{\pi_2(k-1)},\ldots, d_{\pi_2(0)})_b$,
  are \textit{conjugate} if $\pi_1 \tau_1 \pi_1^{-1}=\pi_2 \tau_2 \pi_2^{-1}$.
\end{definition}

Clearly, permutiple conjugacy defines an equivalence relation on the collection of all base-$b$ 
permutiples with multiplier $n$ having the same digits.
From this fact, we need to establish some additional terminology. For any
two  $(n,b, \tau_1)$ and $(n,b, \tau_2)$-permutiples
of the same conjugacy class,
$(d_{\pi_1(k)}, d_{\pi_1(k-1)},\ldots, d_{\pi_1(0)})_b$  and $(d_{\pi_2(k)}, d_{\pi_2(k-1)},\ldots, d_{\pi_2(0)})_b$,
we shall refer to the common permutation $\beta=\pi_1 \tau_1 \pi_1^{-1}=\pi_2 \tau_2 \pi_2^{-1}$
as the \textit{base permutation} of the class. 
We emphasize that the base permutation of a conjugacy class might not necessarily be a digit permutation itself. 

Our next result tells us that two permutiples in the same conjugacy class both have the same set of carries. 

\begin{theorem}
Let
$(d_{\pi_1(k)}, d_{\pi_1(k-1)},\ldots, d_{\pi_1(0)})_b$  and $(d_{\pi_2(k)}, d_{\pi_2(k-1)},\ldots, d_{\pi_2(0)})_b$
be $(n,b, \tau_1)$ and $(n,b, \tau_2)$-permutiples, respectively, from the same conjugacy class
with carries given by $c_k$, $c_{k-1}$, $\ldots$ , $c_0$
 and $\hat{c}_k$, $\hat{c}_{k-1}$, $\ldots$ , $\hat{c}_0$, respectively.
 Then $\hat{c}_j=c_{\pi_1^{-1}\pi_2(j)}$ for all $0 \leq j \leq k$.
 \label{conj_class_carries}
\end{theorem}

\begin{proof}
 By assumption we have both that $(nP_{\tau_1}-I)P_{\pi_1}\textbf{d}=(bP_{\psi}-I)\textbf{c}$ and 
 that $(nP_{\tau_2}-I)P_{\pi_2}\textbf{d}=(bP_{\psi}-I)\hat{\textbf{c}}$.
 Then $P_{\pi_1}(nP_{\pi_1\tau_1\pi_1^{-1}}-I)\textbf{d}=(bP_{\psi}-I)\textbf{c}$ and 
 $P_{\pi_2}(nP_{\pi_2\tau_2\pi_2^{-1}}-I)\textbf{d}=(bP_{\psi}-I)\hat{\textbf{c}}$.
 Since both permutiples are conjugate we have that $P_{\pi_1 \tau_1 \pi_1^{-1}}=P_{\pi_2 \tau_2 \pi_2^{-1}}$.
 It follows that $P_{\pi_1^{-1}}(bP_{\psi}-I)\textbf{c}=P_{\pi_2^{-1}}(bP_{\psi}-I)\hat{\textbf{c}}$.
 Reducing modulo $b$ we have $P_{\pi_1^{-1}}\textbf{c} \equiv P_{\pi_2^{-1}}\hat{\textbf{c}} \mod b$, or
 $P_{\pi_1^{-1}\pi_2}\textbf{c} \equiv \hat{\textbf{c}} \mod b$. Theorem \ref{carries} then implies that
 $\hat{\textbf{c}}=P_{\pi_1^{-1}\pi_2}\textbf{c}$.
\end{proof}

The above theorem gives us the following important result.

\begin{theorem}
Let $p=(d_{\pi_1(k)}, d_{\pi_1(k-1)},\ldots, d_{\pi_1(0)})_b$ be an
$(n,b, \tau_1)$-permutiple with carries $c_k$, $c_{k-1}$, $\ldots$ , $c_0$.
If $(d_{\pi_2(k)}, d_{\pi_2(k-1)},\ldots, d_{\pi_2(0)})_b$ is an $(n,b, \tau_2)$-permutiple
from the same conjugacy class as $p$,
 then $c_{\psi \pi_1^{-1}\pi_2(j)}=c_{\pi_1^{-1}\pi_2\psi(j)}$ for all $0 \leq j \leq k$.
\label{psi_commute}
\end{theorem}

\begin{proof}
 Our first assumption in matrix form is 
 $(n P_{\tau_1}-I)P_{\pi_1}\textbf{d}=(bP_{\psi}-I)\textbf{c}$,
 and by Theorem \ref{conj_class_carries}, our second assumption becomes
  $(n P_{\tau_2}-I)P_{\pi_2}\textbf{d}=(bP_{\psi}-I)P_{\pi_1^{-1}\pi_2}\textbf{c}$.
  Using this second equation, we have
  $(n P_{\pi_2\tau_2\pi_2^{-1}}-I)\textbf{d}=P_{\pi_2^{-1}}(n P_{\tau_2}-I)P_{\pi_2}\textbf{d}
  =P_{\pi_2^{-1}}(bP_{\psi}-I)P_{\pi_1^{-1}\pi_2}\textbf{c}$, which by conjugacy
  gives  $(n P_{\pi_1\tau_1\pi_1^{-1}}-I)\textbf{d}
  =P_{\pi_2^{-1}}(bP_{\psi}-I)P_{\pi_1^{-1}\pi_2}\textbf{c}$.
  Multiplying by $P_{\pi_1}$, we then have
  $(n P_{\tau_1}-I)P_{\pi_1}\textbf{d}
  =P_{\pi_1} P_{\pi_2^{-1}}(bP_{\psi}-I)P_{\pi_1^{-1}\pi_2}\textbf{c}$,
  which by the first relation above becomes
  $(b P_{\psi}-I)\textbf{c}
  =P_{\pi_1^{-1}\pi_2}^{-1}(bP_{\psi}-I)P_{\pi_1^{-1}\pi_2}\textbf{c}$.
  The above reduces to $P_{\pi_1^{-1}\pi_2} P_{\psi}\textbf{c}
  =P_{\psi}P_{\pi_1^{-1}\pi_2}\textbf{c}$, and the proof is complete.
\end{proof}

With the above theorem, we may determine a list of candidate permutations, $\pi$, within a particular conjugacy class.
The next theorem tells us that every item on this list yields a permutiple.

\begin{theorem}
Let $p=(d_{\pi_1(k)}, d_{\pi_1(k-1)},\ldots, d_{\pi_1(0)})_b$ be an
$(n,b, \tau_1)$-permutiple with carries $c_k$, $c_{k-1}$, $\ldots$ , $c_0$.
If $\pi_2$ is a permutation such that $c_{\pi_1^{-1}\pi_2(0)}=0$, and $c_{\psi \pi_1^{-1}\pi_2(j)}=c_{\pi_1^{-1}\pi_2 \psi(j)}$ 
for all $0 \leq j \leq k$, then 
$(d_{\pi_2(k)}, d_{\pi_2(k-1)},\ldots, d_{\pi_2(0)})_b$ is an $(n,b, \tau_2)$-permutiple
from the same conjugacy class as $p$, where $\tau_2=\pi_2^{-1}\pi_1\tau_1\pi_1^{-1}\pi_2$.
\label{psi_commute_2}
\end{theorem}

\begin{proof}
 By assumption, we have both that $(nP_{\tau_1}-I)P_{\pi_1}\textbf{d}=(bP_{\psi}-I)\textbf{c}$ and
 $P_{\pi_1^{-1}\pi_2} P_{\psi}\textbf{c}=P_{\psi}P_{\pi_1^{-1}\pi_2}\textbf{c}$.
 Multiplying both sides of the first of the above equations by $P_{\pi_1^{-1}\pi_2}$ yields
 $P_{\pi_1^{-1}\pi_2} (nP_{\tau_1}-I)P_{\pi_1}\textbf{d}= (bP_{\psi}-I)P_{\pi_1^{-1}\pi_2}\textbf{c}$
 using our second assumption. Our first assumption and a routine calculation then shows that
 $(nP_{\pi_2^{-1}\pi_1\tau_1\pi_1^{-1}\pi_2}-I)P_{\pi_2}\textbf{d}=(bP_{\psi}-I)P_{\pi_1^{-1}\pi_2}\textbf{c}$,
 and the proof is complete.
\end{proof}

If we take the point of view that there is a reference permutiple, $p$, in every conjugacy class 
(not necessarily our initial example),
we can simply take $\pi_1$ in the above to be the identity, and $\pi=\pi_2$ to be any suitable permutation of the digits of $p$.
The above theorem then implies that $P_{\pi}P_{\psi}\textbf{c}=P_{\psi}P_{\pi}\textbf{c}$, 
or, $P_{\psi}\textbf{c}=P_{\pi\psi\pi^{-1}}\textbf{c}$.
We state the above as a single corollary to Theorems \ref{psi_commute} and \ref{psi_commute_2}.

\begin{corollary}
Let $p=(d_{k}, d_{k-1},\ldots, d_{0})_b$ be an
$(n,b, \tau_1)$-permutiple with carries $c_k$, $c_{k-1}$, $\ldots$ , $c_0$. 
Then, the following hold:
\begin{enumerate}
 \item If $(d_{\pi(k)}, d_{\pi(k-1)},\ldots, d_{\pi(0)})_b$ is an $(n,b, \tau_2)$-permutiple
from the same conjugacy class as $p$,
 then $c_{\psi(j)}=c_{\pi \psi \pi^{-1}(j)}$ for all $0 \leq j \leq k$.
 \item If $\pi$ is a permutation such that $c_{\pi(0)}=0$ and $c_{\psi(j)}=c_{\pi \psi \pi^{-1}(j)}$
 for all $0 \leq j \leq k$, then $(d_{\pi(k)}, d_{\pi(k-1)},\ldots, d_{\pi(0)})_b$ is an $(n,b, \tau_2)$-permutiple
from the same conjugacy class as $p$.
\end{enumerate} 
\label{cycle_theorem}
\end{corollary}

The above corollary, together with $\pi\psi\pi^{-1}=(\pi(0),\pi(1),\ldots,\pi(k))$, enables us to construct the entire collection
of all permutations, $\pi$, of a reference permutiple's digits within a particular conjugacy class. 
The next example illustrates this process. 

\begin{example}
Consider an earlier example, $(d_4,d_3,d_2,d_1,d_0)=(4, 3, 5, 1, 2)_6=2\cdot(2, 1, 5, 3, 4)_6$, a $(2,6,\sigma)$-permutiple 
whose carry vector is $(0,1,1,1,0)$.  The reader will note that this example is a $(2,6)$-palintiple.

We shall use Corollary \ref{cycle_theorem} to find all permutiples conjugate to 
$(d_4,d_3,d_2,d_1,d_0)$. 
By Corollary \ref{cycle_theorem}, we know that if $(d_4,d_3,d_2,d_1,d_0)$ and
$(d_{\pi(4)},d_{\pi(2)},d_{\pi(2)},d_{\pi(1)},d_{\pi(0)})$ are conjugate, then $\pi$ necessarily satisfies
$$
\left[
\begin{array}{c}
1\\
1\\
1\\
0\\
0\\
\end{array}
\right]
=
P_{\psi}
\left[
\begin{array}{c}
0\\
1\\
1\\
1\\
0\\
\end{array}
\right]
=
P_{\pi \psi \pi^{-1}}
\left[
\begin{array}{c}
0\\
1\\
1\\
1\\
0\\
\end{array}
\right].
$$
Restating the above in component form, we have
$$
\begin{array}{l}
c_{\pi \psi \pi^{-1}(0)}=1=c_1, c_2, \mbox{ or, } c_3,\\
c_{\pi \psi \pi^{-1}(1)}=1=c_1, c_2, \mbox{ or, } c_3,\\
c_{\pi \psi \pi^{-1}(2)}=1=c_1, c_2, \mbox{ or, } c_3,\\
c_{\pi \psi \pi^{-1}(3)}=0=c_0 \mbox{ or } c_4,\\
c_{\pi \psi \pi^{-1}(4)}=0=c_0 \mbox{ or } c_4.\\
\end{array}
$$
The list of possible candidates for $\pi \psi \pi^{-1}$ from above then consists of any permutation which can be constructed
from
$
\left(
\begin{array}{ccccc}
0 & 1 & 2 & 3 & 4\\
1,2 \mbox{ or, } 3 & 1,2 \mbox{ or, } 3  & 1,2 \mbox{ or, } 3 & 0 \mbox{ or } 4 & 0 \mbox{ or } 4 \\
\end{array}
\right),
$
and since $\pi \psi \pi^{-1}$ is a 5-cycle, we can eliminate any fixed points so that the above permutation must have the form
$$
\left(
\begin{array}{ccccc}
0                  & 1                & 2               & 3 & 4\\
1,2 \mbox{ or, } 3 & 2 \mbox{ or } 3  & 1 \mbox{ or } 3 & 4 & 0 \\
\end{array}
\right).
$$
Therefore, $\pi \psi \pi^{-1}=(\pi(0),\pi(1),\pi(2),\pi(3),\pi(4))$ must equal either $(0,1,2,3,4)$ or $(0,2,1,3,4)$.
Since $c_0=c_4=0$, we have four permutations which satisfy the conditions of Corollary \ref{cycle_theorem}: 
the identity
$
\varepsilon=\left(
\begin{array}{ccccc}
0 & 1 & 2 & 3 & 4 \\
0 & 1 & 2 & 3 & 4 \\
\end{array}
\right)
$,
$
\psi^4=\left(
\begin{array}{ccccc}
0 & 1 & 2 & 3 & 4 \\
4 & 0 & 1 & 2 & 3 \\
\end{array}
\right)
$,
$
(1,2)=\left(
\begin{array}{ccccc}
0 & 1 & 2 & 3 & 4 \\
0 & 2 & 1 & 3 & 4 \\
\end{array}
\right)
$, and
$
(1,2)\psi^4=\left(
\begin{array}{ccccc}
0 & 1 & 2 & 3 & 4 \\
4 & 0 & 2 & 1 & 3 \\
\end{array}
\right)
$.
These permutations give us the entire conjugacy class listed in the table below.
\begin{center}
 \begin{tabular}{|c|c|c|c|}
\hline
$(d_{\pi(4)},d_{\pi(3)},d_{\pi(2)},d_{\pi(1)},d_{\pi(0)})_6$ & $\pi$ & $\tau$ & $(c_{\pi(4)},c_{\pi(3)},c_{\pi(2)},c_{\pi(1)},c_{\pi(0)})$\\\hline
$(4,3,5,1,2)_6$ & $\varepsilon$ & $\rho$ & $(0,1,1,1,0)$\\\hline
$(4,3,1,5,2)_6$ & $(1,2)$ & $(1,2)\rho(1,2)$ & $(0,1,1,1,0)$\\\hline
$(3,5,1,2,4)_6$ & $\psi^4$ & $\psi^{-4}\rho\psi^4$ & $(1,1,1,0,0)$\\\hline
$(3,1,5,2,4)_6$ & $(1,2)\psi^4$ & $\psi^{-4}(1,2)\rho(1,2)\psi^4$ & $(1,1,1,0,0)$\\\hline
\end{tabular}
\end{center}

\end{example}

We shall now look between conjugacy classes.
The converse of Theorem \ref{conj_class_carries} does not hold in general. However, assuming its consequent
does yield a useful theorem which gives us a list of base permutation candidates from every conjugacy class
with the same set of carries as the original example.

\begin{theorem}
Let 
$(d_{\pi_1(k)}, d_{\pi_1(k-1)},\ldots, d_{\pi_1(0)})_b$  and $(d_{\pi_2(k)}, d_{\pi_2(k-1)},\ldots, d_{\pi_2(0)})_b$
be $(n,b, \tau_1)$ and $(n,b, \tau_2)$-permutiples, respectively,
with carries given by $c_k$, $c_{k-1}$, $\ldots$ , $c_0$
and $\hat{c}_k$, $\hat{c}_{k-1}$, $\ldots$ , $\hat{c}_0$, respectively.
If $\hat{c}_j=c_{\pi_1^{-1}\pi_2(j)}$ for all $0 \leq j \leq k$, then 
$n d_{\pi_1\tau_1\pi_1^{-1}(j)} \equiv n d_{\pi_2\tau_2\pi_2^{-1}(j)} \mod b$
for all $0 \leq j \leq k$.
\end{theorem}

\begin{proof}
 Our assumptions in matrix form are
 $(nP_{\tau_1}-I)P_{\pi_1}\textbf{d}=(bP_{\psi}-I)\textbf{c}$
 and  $(nP_{\tau_2}-I)P_{\pi_2}\textbf{d}=(bP_{\psi}-I)P_{\pi_1^{-1}\pi_2}\textbf{c}$.
 Reducing modulo $b$, we have both
 $(nP_{\tau_1}-I)P_{\pi_1}\textbf{d} \equiv -\textbf{c} \mod b$
 and  $(nP_{\tau_2}-I)P_{\pi_2}\textbf{d}\equiv-P_{\pi_1^{-1}\pi_2}\textbf{c} \mod b$.
 It follows that $P_{\pi_1}(nP_{\pi_1\tau_1\pi_1^{-1}}-I)\textbf{d}\equiv-\textbf{c} \mod b$
 and  $P_{\pi_2}(nP_{\pi_2\tau_2\pi_2^{-1}}-I)\textbf{d}\equiv-P_{\pi_1^{-1}\pi_2}\textbf{c} \mod b$,
 from which we obtain 
 $(nP_{\pi_1\tau_1\pi_1^{-1}}-I)\textbf{d}
 \equiv-P_{\pi_1^{-1}}\textbf{c} 
 \equiv (nP_{\pi_2\tau_2\pi_2^{-1}}-I)\textbf{d}\mod b$.
 Thus, $nP_{\pi_1\tau_1\pi_1^{-1}}\textbf{d} \equiv nP_{\pi_2\tau_2\pi_2^{-1}}\textbf{d}\mod b$.
\end{proof}

Letting $\pi_1$ be the identity and $\tau_1$ be $\sigma$
in the above theorem, we obtain a result which relates any $(n,b,\tau)$-permutiple with the same set of carries
to our known example. 

\begin{corollary}
Let $(d_k, \ldots, d_0)_b$ be an $(n,b,\sigma)$-permutiple with carries $c_k$, $\ldots$ , $c_0$, and let 
 $(d_{\pi(k)}, d_{\pi(k-1)},\ldots, d_{\pi(0)})_b$ be an $(n,b,\tau)$-permutiple
with carries $c_{\pi(k)}$, $c_{\pi(k-1)}$, $\ldots$ , $c_{\pi(0)}$, 
then $nd_{\sigma (j)}\equiv nd_{\pi \tau \pi^{-1}(j)} \mod b$ for all $0\leq j\leq k$.
\label{base_perm}
\end{corollary} 

The above corollary will, under certain conditions, enable us to find all 
base permutations $\beta=\pi\tau\pi^{-1}$ for every possible conjugacy class.

Our next result gives us conditions for the existence of a bijective correspondence 
between permutiples, namely,
$(d_{\pi(k)}, d_{\pi(k-1)},\ldots, d_{\pi(0)})_b \mapsto (d_{\alpha\pi (k)}, d_{\alpha \pi (k-1)},\ldots, d_{\alpha \pi(0)})_b$.

\begin{theorem}
Let $(d_{\pi_1(k)}, d_{\pi_1(k-1)},\ldots, d_{\pi_1(0)})_b$  and $(d_{\pi_2(k)}, d_{\pi_2(k-1)},\ldots, d_{\pi_2(0)})_b$
be $(n,b, \tau_1)$ and $(n,b, \tau_2)$-permutiples, respectively. 
If there exists an $\alpha$ such that
$(nP_{\tau_2}-I)P_{\alpha}\textbf{d}=(nP_{\tau_1}-I)\textbf{d}$,
then $(d_{\pi(k)}, d_{\pi(k-1)},\ldots, d_{\pi(0)})_b$ is an $(n,b,\pi^{-1}\tau_1\pi)$-permutiple
if and only if
$(d_{\alpha\pi (k)}, d_{\alpha \pi (k-1)},\ldots, d_{\alpha \pi(0)})_b$ is an 
$(n,b, \pi^{-1}\tau_2\pi)$-permutiple.
\label{bijection}
\end{theorem}

\begin{proof}
Suppose 
$(d_{\pi(k)}, d_{\pi(k-1)},\ldots, d_{\pi(0)})_b$ is an $(n,b,\pi^{-1}\tau_1\pi)$-permutiple with 
carries $c_k$, $c_{k-1}$, $\ldots$ , $c_0$, then
$(nP_{\pi^{-1}\tau_1\pi}-I)P_{\pi}\textbf{d}=(bP_{\psi}-I)\textbf{c}$.
Now, by the theorem hypothesis, we have
$
(nP_{\pi^{-1}\tau_1\pi}-I)P_{\pi}\textbf{d}
=P_{\pi}(nP_{\tau_1}-I)\textbf{d}
=P_{\pi}(nP_{\tau_2}-I)P_{\alpha}\textbf{d}
=(nP_{\pi^{-1}\tau_2\pi}-I)P_{\alpha\pi}\textbf{d}
$
so that $(nP_{\pi^{-1}\tau_2\pi}-I)P_{\alpha\pi}\textbf{d}=(bP_{\psi}-I)\textbf{c}$.
By Theorem \ref{fund_conv}, the forward implication holds. 
The reverse implication follows in similar fashion.
\end{proof}

Thus, the above gives us a bijection between conjugacy classes provided that 
$\pi_1 \tau_1 \pi_1^{-1} \neq \pi_2 \tau_2 \pi_2^{-1}$.
Also, the reader should note that the carries of $(d_{\pi(k)}, d_{\pi(k-1)},\ldots, d_{\pi(0)})_b$ and 
$(d_{\alpha\pi (k)}, d_{\alpha \pi (k-1)},\ldots, d_{\alpha \pi(0)})_b$ must be the same 
according to the above argument.

At this point the big question is whether or not the results given thus far can give us 
all the examples we seek. If we recall our initial examples which motivated our
conjugacy class definition, we can see that, in general, the answer is no. 
However, there is a condition which guarantees that we have found all of the desired examples. 
The next theorem tells us that if $n$ divides $b$, then all permutiples
having the same digits as our known example have the same set of carries as 
the known example. 

\begin{theorem}
Let $(d_k, \ldots, d_0)_b$ be an $(n,b,\sigma)$-permutiple with carries $c_k$, $\ldots$ , $c_0$, and let 
 $(d_{\pi(k)}, d_{\pi(k-1)},\ldots, d_{\pi(0)})_b$ be an $(n,b,\tau)$-permutiple
with carries $\hat{c}_k$, $\hat{c}_{k-1}$, $\ldots$ , $\hat{c}_0$. 
If $n$ divides $b$, then $\hat{c}_j=c_{\pi(j)}$ for all $0 \leq j \leq k$. 
\label{n_div_b}
\end{theorem}

\begin{proof}
 By Equation \ref{matrix}, we have both $(b P_\psi-I)\textbf{c} = (nP_\sigma-I) \textbf{d}$ and
 $(b P_\psi-I)\hat{\textbf{c}} = (nP_\tau-I) P_{\pi}\textbf{d}$. Since $n$ divides $b$, it follows that
 $\textbf{c} \equiv \textbf{d} \mod n$ and $\hat{\textbf{c}} \equiv P_{\pi}\textbf{d} \mod n$.
Thus, $\hat{\textbf{c}} \equiv P_{\pi}\textbf{c} \mod n$. By Theorem \ref{carries}, it follows that
$\hat{\textbf{c}} = P_{\pi}\textbf{c}$, which establishes the theorem.
\end{proof}

Thus, when $n$ divides $b$, Corollary \ref{base_perm} enables us to find
every possible base permutation $\beta=\pi \tau \pi^{-1}$.
From there, we may use either Theorem \ref{fund} or other techniques (such as those in the following example) to find a reference
permutiple from each conjugacy class.
Then, using Corollary \ref{cycle_theorem} we find all $\pi$ within each
conjugacy class. Thus, when $n$ divides $b$, finding all permutiples with the same digits as a known permutiple
becomes considerably easier. The next two examples illustrate the above approach.

\begin{example}
 Using our results, we shall find all possible 5-digit $(2,6,\sigma)$-permutiples starting from the
 $(2,6,\rho)$-permutiple $(d_4,d_3,d_2,d_1,d_0)_6=(4,3,5,1,2)_6=2\cdot(2,1,5,3,4)_6$ 
 with carries $(c_4,c_3,c_2,c_1,c_0)=(0,1,1,1,0)$ which we considered in the previous example. 
 
By Corollary \ref{base_perm} and Theorem \ref{n_div_b},
any suitable base permutation $\beta$ necessarily satisfies $2d_{\rho(j)}\equiv 2d_{\beta(j)} \mod 6$ for all $0 \leq j \leq 4$,
which, since 2 divides 6, becomes $d_{\rho(j)}\equiv d_{\beta(j)} \mod 3$ for all $0 \leq j \leq 4$.
In matrix form we have
$$ 
\left[\begin{matrix}
    d_{\beta(0)}\\
   d_{\beta(1)}\\
   d_{\beta(2)}\\
   d_{\beta(3)}\\
   d_{\beta(4)}\\
 \end{matrix}
\right] 
\equiv
 \left[
 \begin{matrix}
   1\\ 
   0\\
   2\\
   1\\
    2\\
 \end{matrix}
\right]
\mod 3.
$$
Expressing the above in component form gives us
$$
\begin{array}{l}
d_{\beta(0)}=1=d_1 \mbox{ or } d_4,\\
d_{\beta(1)}=0=d_3, \\
d_{\beta(2)}=2=d_0 \mbox{ or } d_2,\\
d_{\beta(3)}=1=d_1 \mbox{ or } d_4,\\
d_{\beta(4)}=2=d_0 \mbox{ or } d_2. \\
\end{array}
$$ 
Then, any base permutation, $\beta$, necessarily has the form
$
\left(
\begin{array}{ccccc}
0 & 1 & 2 & 3 & 4 \\
1 \mbox{ or } 4 & 3  & 0 \mbox{ or } 2 & 1 \mbox{ or } 4 & 0 \mbox{ or } 2 \\
\end{array}
\right).
$ 
Thus, there are four candidate base permutations: $\beta_1=\rho$, $\beta_2=(4,2,0,1,3)$, $\beta_3=(4,2,0)(1,3)$, and 
$\beta_4=(4,0,1,3)$.

For $\beta_1=\rho$, the solution corresponding to our known example, we note that we already determined its conjugacy
class in the previous example.

We now consider the conjugacy class for $\beta_2=(4,2,0,1,3)$. 
We shall find this class by finding a bijection from 
the above class with base permutation $\rho$ as per Theorem \ref{bijection}. 
But first we must find an example from the conjugacy class with base permutation $\beta_2$.
Provided a suitable permutation $\alpha$ exists, 
the bijection guaranteed by Theorem \ref{bijection} maps permutiples
with carry vector $\textbf{c}$ to other permutiples with the same carry vector.
Therefore, if such an $\alpha$ exists, we know that our known example $(d_k,\ldots,d_0)$ will map to
$(d_{\alpha(k)}, d_{\alpha(k-1)},\ldots, d_{\alpha(0)})_b$. Then, by Theorem \ref{n_div_b},
$\alpha$ fixes the carry vector $(0,1,1,1,0)$ of our known $(2,6,\rho)$ example. 
Thus, $\alpha$ must contain a factor of either the identity or $(4,0)$, and a factor
of either $(3,2,1)$, $(1,2,3)$, $(3,1,2)$, $(1,2)$, $(1,2)$, $(1,3)$, or $(2,3)$.
Checking the possibilities which are not already listed in the class found in the previous example by simple base-6 arithmetic,
we see that either $\alpha=(4,0)(3,2,1)=(1,2)\rho$ or $\alpha=(4,0)(3,2)=(1,2)\rho(1,2)$.
Respectively, these values give us the permutiples $(2,5,1,3,4)_6=2\cdot(1,2,3,4,5)_6$,
for which $\tau_2=(0,1)\beta_2(0,1)=(4,2,1,0,3)$,
and $(2,5,3,1,4)_6=2\cdot(1,2,4,3,5)_6$, for which $\tau_2=(1,2)(0,1)\beta_2(0,1)(1,2)=(4,1,2,0,3)$.
The reader may check that both values of $\alpha$ yield a bijection. 
Using the first value, $\alpha=(1,2)\rho$ with $\tau_2=(0,1)\beta_2(0,1)=(4,2,1,0,3)$,
Theorem \ref{bijection} easily gives us the rest of the permutiples in this class
and may be found in the table below. 

\begin{center}
 \begin{tabular}{|c|c|c|c|}
\hline
$(d_{\pi(4)},d_{\pi(3)},d_{\pi(2)},d_{\pi(1)},d_{\pi(0)})_6$ & $\pi$ & $\tau$ & $(c_{\pi(4)},c_{\pi(3)},c_{\pi(2)},c_{\pi(1)},c_{\pi(0)})$\\\hline
$(2,5,1,3,4)_6$ & $(1,2)\rho$ & $\tau_2=(4,2,1,0,3)$ & $(0,1,1,1,0)$\\\hline
$(2,5,3,1,4)_6$ & $(1,2)\rho(1,2)$ & $(1,2)\tau_2(1,2)$ & $(0,1,1,1,0)$\\\hline
$(5,1,3,4,2)_6$ & $(1,2)\rho \psi^4$ & $\psi^{-4}\tau_2\psi^4$ & $(1,1,1,0,0)$\\\hline
$(5,3,1,4,2)_6$ & $(1,2)\rho(1,2)\psi^4$ & $\psi^{-4}(1,2)\tau_2(1,2)\psi^4$ & $(1,1,1,0,0)$\\\hline
\end{tabular}
\end{center}

Determining the conjugacy class having the base permutation $\beta_3=(4,2,0)(3,1)$, 
we find an example from this class.
We shall attempt to find a $(2,6,\beta_3)$-permutiple 
$(d_{\pi(k)}, d_{\pi(k-1)},\ldots, d_{\pi(0)})_b$.
Then $\pi\beta_3\pi^{-1}=\beta_3$, so that $(\pi(4),\pi(2),\pi(0))(\pi(3),\pi(1))=(4,2,0)(3,1)$.
Since the first carry of any permutiple must always be zero, we know that $\pi(0)$ equals
either 0 or 4. Hence, provided $\pi$ commutes with $\beta_3$, there are 8 possibilities,
among which, $\pi=\beta_3$ provides us with a solution.

Therefore, by Theorem \ref{n_div_b}, with $\pi=(4,2,0)(1,3)$, we have that 
$(c_{\pi(4)},c_{\pi(3)},c_{\pi(2)},c_{\pi(1)},c_{\pi(0)})=(1,1,0,1,0)$.
So the new initial carry vector is given by
$\hat{\textbf{c}}=P_{\pi}\textbf{c}=\left[\begin{matrix}
   0\\
   1\\
   0\\
   1\\
   1\\
 \end{matrix}
\right].$
For notational convenience, we shall write $(\hat{d}_4,\hat{d}_3,\hat{d}_2,\hat{d}_1,\hat{d}_0)_6$
in place of $(d_{\pi(4)},d_{\pi(2)},d_{\pi(2)},d_{\pi(1)},d_{\pi(0)})_6=(5,1,2,3,4)_6$.
Then, by Corollary \ref{cycle_theorem}, if  
$(\hat{d}_4,\hat{d}_3,\hat{d}_2,\hat{d}_1,\hat{d}_0)_6$ and
$(\hat{d}_{\pi(4)},\hat{d}_{\pi(2)},\hat{d}_{\pi(2)},\hat{d}_{\pi(1)},\hat{d}_{\pi(0)})_6$ are conjugate, 
then $\pi$ necessarily satisfies
$$
\left[
\begin{array}{c}
1\\
0\\
1\\
1\\
0\\
\end{array}
\right]
=
P_{\psi}
\left[
\begin{array}{c}
0\\
1\\
0\\
1\\
1\\
\end{array}
\right]
=
P_{\pi \psi \pi^{-1}}
\left[
\begin{array}{c}
0\\
1\\
0\\
1\\
1\\
\end{array}
\right].
$$
As with the previous example, the above is restated in component form,
$$
\begin{array}{l}
\hat{c}_{\pi \psi \pi^{-1}(0)}=1=\hat{c}_1, \hat{c}_3, \mbox{ or, } \hat{c}_4,\\
\hat{c}_{\pi \psi \pi^{-1}(1)}=0=\hat{c}_0 \mbox{ or } \hat{c}_2,\\
\hat{c}_{\pi \psi \pi^{-1}(2)}=1=\hat{c}_1, \hat{c}_3, \mbox{ or, } \hat{c}_4,\\
\hat{c}_{\pi \psi \pi^{-1}(3)}=1=\hat{c}_1, \hat{c}_3, \mbox{ or, } \hat{c}_4,\\
\hat{c}_{\pi \psi \pi^{-1}(4)}=0=\hat{c}_0 \mbox{ or } \hat{c}_2,\\
\end{array}
$$
which provides a list of possible values of $\pi \psi \pi^{-1}$, that is, any permutation of the form
$
\left(
\begin{array}{ccccc}
0 & 1 & 2 & 3 & 4\\
1,3 \mbox{ or, } 4 & 0 \mbox{ or } 2  & 1,3 \mbox{ or, } 4 & 1 \mbox{ or } 4 & 0 \mbox{ or } 2 \\
\end{array}
\right).
$
Examining all 8 possibilities, $\pi \psi \pi^{-1}=(\pi(0),\pi(1),\pi(2),\pi(3),\pi(4))$ must equal either 
$(0,1,2,3,4)$, $(0,3,4,2,1)$, $(0,4,2,3,1)$, or $(0,3,1,2,4)$. 
Since $\hat{c}_0=\hat{c}_2=0$, 
we have the following 8 permutations which satisfy the conditions of Corollary \ref{cycle_theorem}: the identity
$
\varepsilon=\left(
\begin{array}{ccccc}
0 & 1 & 2 & 3 & 4 \\
0 & 1 & 2 & 3 & 4 \\
\end{array}
\right)
$,
$
\pi_2=\left(
\begin{array}{ccccc}
0 & 1 & 2 & 3 & 4 \\
0 & 3 & 4 & 2 & 1 \\
\end{array}
\right)
$,
$
\pi_3=\left(
\begin{array}{ccccc}
0 & 1 & 2 & 3 & 4 \\
0 & 4 & 2 & 3 & 1 \\
\end{array}
\right)
$,
$
\pi_4=\left(
\begin{array}{ccccc}
0 & 1 & 2 & 3 & 4 \\
0 & 3 & 1 & 2 & 4 \\
\end{array}
\right)
$,
$
\pi_5=\left(
\begin{array}{ccccc}
0 & 1 & 2 & 3 & 4 \\
2 & 3 & 4 & 0 & 1 \\
\end{array}
\right)
$,
$
\pi_6=\left(
\begin{array}{ccccc}
0 & 1 & 2 & 3 & 4 \\
2 & 1 & 0 & 3 & 4\\
\end{array}
\right)
$,
$
\pi_7=\left(
\begin{array}{ccccc}
0 & 1 & 2 & 3 & 4 \\
2 & 3 & 1 & 0 & 4 \\
\end{array}
\right)
$, and
$
\pi_8=\left(
\begin{array}{ccccc}
0 & 1 & 2 & 3 & 4 \\
2 & 4 & 0 & 3 & 1 \\
\end{array}
\right)
$.
These permutations give us the entire conjugacy class given by the table below.
\begin{center}
 \begin{tabular}{|c|c|c|c|}
\hline
$(\hat{d}_{\pi(4)},\hat{d}_{\pi(3)},\hat{d}_{\pi(2)},\hat{d}_{\pi(1)},\hat{d}_{\pi(0)})_6$ & $\pi$ & $\tau$ & $(\hat{c}_{\pi(4)},\hat{c}_{\pi(3)},\hat{c}_{\pi(2)},\hat{c}_{\pi(1)},\hat{c}_{\pi(0)})$\\\hline
$(5,1,2,3,4)_6$ & $\varepsilon$ & $\beta_3=(4,2,0)(1,3)$ & $(1,1,0,1,0)$\\\hline
$(3,2,5,1,4)_6$ & $\pi_2$ & $\pi_2^{-1}\beta_3\pi_2$ &$(1,0,1,1,0)$\\\hline 
$(3,1,2,5,4)_6$ & $\pi_3$ & $\pi_3^{-1}\beta_3\pi_3$ & $(1,1,0,1,0)$\\\hline 
$(5,2,3,1,4)_6$ & $\pi_4$ & $\pi_4^{-1}\beta_3\pi_4$ &$(1,0,1,1,0)$\\\hline 
$(3,4,5,1,2)_6$ & $\pi_5$ & $\pi_5^{-1}\beta_3\pi_5$ &$(1,0,1,1,0)$\\\hline 
$(5,1,4,3,2)_6$ & $\pi_6$ & $\pi_6^{-1}\beta_3\pi_6$ & $(1,1,0,1,0)$\\\hline 
$(5,4,3,1,2)_6$ & $\pi_7$ & $\pi_7^{-1}\beta_3\pi_7$ &$(1,0,1,1,0)$\\\hline 
$(3,1,4,5,2)_6$ & $\pi_8$ & $\pi_8^{-1}\beta_3\pi_8$ & $(1,1,0,1,0)$\\\hline 
\end{tabular}
\end{center}

We now consider the final candidate $\beta_4=(4,0,1,3)$.
By Theorem \ref{n_div_b}, any suitable $\pi$ in this class must satisfy
$P_{\pi}(2P_{\beta_4}-I)\textbf{d}=(2P_{\pi^{-1}\beta_4\pi}-I)P_{\pi}\textbf{d}=(6P_{\psi}-I)P_{\pi}\textbf{c}$, or,
$$
P_{\pi}(2P_{\beta_4}-I)\left[\begin{matrix}
    2\\
   1\\
   5\\
   3\\
   4\\
 \end{matrix}
 \right]
 =
(6P_{\psi}-I)
P_{\pi}\left[\begin{matrix}
    0\\
   1\\
   1\\
   1\\
   0\\
 \end{matrix}
\right],$$
which simplifies to
$$
P_{\pi}\left[\begin{matrix}
    0\\
   1\\
   1\\
   1\\
   0\\
 \end{matrix}
 \right]
 =
P_{\psi}
P_{\pi}\left[\begin{matrix}
    0\\
   1\\
   1\\
   1\\
   0\\
 \end{matrix}
 \right].
$$
We plainly see that $P_{\psi}$ must fix the column vector $P_{\pi}\left[\begin{matrix}
    0\\
   1\\
   1\\
   1\\
   0\\
 \end{matrix}
 \right].$
Since there is no permutation, $\pi$, which makes this statement true,
there is no $\pi$ for which $(d_{\pi(k)}, d_{\pi(k-1)},\ldots, d_{\pi(0)})_b$
is a $(2,6,\pi^{-1}\beta_4\pi)$-permutiple. 
\end{example}

We shall now apply our techniques to another example with more digits and a larger base. 
However, instead of beginning with a palintiple as in the previous examples, we shall find all permutiples
having the same digits as the base-12 cyclic number $(1, 8, 6, 10, 3, 5)_{12}$. 
In particular, we shall find all $(3,12)$-permutiples from the example 
$(5, 1, 8, 6, 10, 3)_{12} = 3 \cdot (1, 8, 6, 10, 3, 5)_{12}$ with carries $(c_5,c_4,c_3,c_2,c_1,c_0)=(2, 1, 2, 0, 1, 0)$. 

\begin{example}
 We begin by computing the conjugacy class containing 
 $(d_5,d_4,d_3,d_2,d_1,d_0)=(5, 1, 8, 6, 10, 3)_{12} = 3 \cdot (1, 8, 6, 10, 3, 5)_{12}$.
 The carries are given by $(c_5,c_4,c_3,c_2,c_1,c_0)=(2, 1, 2, 0, 1, 0)$, so 
 by Corollary \ref{cycle_theorem}, any permutation, $\pi$, of the digits of our initial example must satisfy
 $$
 \left[
\begin{array}{c}
1\\
0\\
2\\
1\\
2\\
0\\
\end{array}
\right]
=
P_{\pi \psi \pi^{-1}}
\left[
\begin{array}{c}
0\\
1\\
0\\
2\\
1\\
2\\
\end{array}
\right].
 $$
Translating the above to component form, we have
$$
\begin{array}{l}
c_{\pi \psi \pi^{-1}(0)}=1=c_1 \mbox{ or } c_4,\\
c_{\pi \psi \pi^{-1}(1)}=0=c_0 \mbox{ or } c_2,\\
c_{\pi \psi \pi^{-1}(2)}=2=c_3 \mbox{ or } c_5,\\
c_{\pi \psi \pi^{-1}(3)}=1=c_1 \mbox{ or } c_4,\\
c_{\pi \psi \pi^{-1}(4)}=2=c_3 \mbox{ or } c_5,\\
c_{\pi \psi \pi^{-1}(5)}=0=c_0 \mbox{ or } c_2.\\
\end{array}
$$ 
Then, it must be that $\pi \psi \pi^{-1}$ has the form
$
\left(
\begin{array}{cccccc}
0 & 1 & 2 & 3 & 4 & 5\\
1 \mbox{ or } 4 & 0 \mbox{ or } 2  & 3 \mbox{ or } 5 & 4 \mbox{ or } 1 & 5 \mbox{ or } 3 & 0 \mbox{ or } 2 \\
\end{array}
\right).
$
From the above, there are 3 possible 6-cycles which yield acceptable permutations 
$\pi \psi \pi^{-1}=(\pi(0),\pi(1),\pi(2),\pi(3),\pi(4),\pi(5))$, namely,
$(0,1,2,3,4,5)$, $(0,4,5,2,3,1)$, and $(0,4,3,1,2,5)$.
Since $c_0=c_2=0$, we have six values of $\pi$, namely, the identity $\varepsilon$, $\pi_2=\psi^2$,
$\pi_3=
\left(
\begin{array}{cccccc}
0 & 1 & 2 & 3 & 4 & 5 \\
0 & 4 & 5 & 2 & 3 & 1\\
\end{array}
\right)$,
$\pi_4=
\left(
\begin{array}{cccccc}
0 & 1 & 2 & 3 & 4 & 5 \\
2 & 3 & 1 & 0 & 4 & 5\\
\end{array}
\right)$,
$\pi_5=
\left(
\begin{array}{cccccc}
0 & 1 & 2 & 3 & 4 & 5 \\
0 & 4 & 3 & 1 & 2 & 5\\
\end{array}
\right)$, and
$\pi_6=
\left(
\begin{array}{cccccc}
0 & 1 & 2 & 3 & 4 & 5 \\
2 & 5 & 0 & 4 & 3 & 1\\
\end{array}
\right)$,
which yield the following permutiple conjugacy class.
\begin{center}
 \begin{tabular}{|c|c|c|c|}
\hline
$(d_{\pi(5)},d_{\pi(4)},d_{\pi(3)},d_{\pi(2)},d_{\pi(1)},d_{\pi(0)})_{12}$ & $\pi$ & $\tau$ & $(c_{\pi(5)},c_{\pi(4)},c_{\pi(3)},c_{\pi(2)},c_{\pi(1)},c_{\pi(0)})$\\\hline
$(5,1,8,6,10,3)_{12}$ & $\varepsilon$ & $\psi^5$ & $(2,1,2,0,1,0)$\\\hline
$(10,3,5,1,8,6)_{12}$ & $\pi_2=\psi^2$ & $\psi^{-2}\psi^5\psi^{2}=\psi^5$ &$(1,0,2,1,2,0)$\\\hline 
$(10,8,6,5,1,3)_{12}$ & $\pi_3$ & $\pi_3^{-1} \psi^5 \pi_3$ & $(1,2,0,2,1,0)$\\\hline 
$(5,1,3,10,8,6)_{12}$ & $\pi_4$ & $\pi_4^{-1} \psi^5 \pi_4$ &$(2,1,0,1,2,0)$\\\hline 
$(5,6,10,8,1,3)_{12}$ & $\pi_5$ & $\pi_5^{-1} \psi^5 \pi_5$ &$(2,0,1,2,1,0)$\\\hline 
$(10,8,1,3,5,6)_{12}$ & $\pi_6$ & $\pi_6^{-1} \psi^5 \pi_6$ & $(1,2,1,0,2,0)$\\\hline 
\end{tabular}
\end{center}

We now set to the task of finding all candidate base permutations. 
By Corollary \ref{base_perm}, we have for any base permutation, $\beta$, that
$$ 
\left[\begin{matrix}
    d_{\beta(0)}\\
   d_{\beta(1)}\\
   d_{\beta(2)}\\
   d_{\beta(3)}\\
   d_{\beta(4)}\\
   d_{\beta(4)}\\   
 \end{matrix}
\right] 
\equiv
 \left[
 \begin{matrix}
   1\\
   3\\
   2\\
   2\\
    0\\
    1\\
 \end{matrix}
\right]
\mod 4.
$$
The above expressed in component form gives us
$$
\begin{array}{l}
d_{\beta(0)}=1=d_4 \mbox{ or } d_5,\\
d_{\beta(1)}=3=d_0, \\
d_{\beta(2)}=2=d_1 \mbox{ or } d_2,\\
d_{\beta(3)}=2=d_1 \mbox{ or } d_2,\\
d_{\beta(4)}=0=d_3, \\
d_{\beta(5)}=1=d_4 \mbox{ or } d_5.\\
\end{array}
$$ 
Then, a base permutation, $\beta$, has the form
$
\left(
\begin{array}{cccccc}
0 & 1 & 2 & 3 & 4 & 5\\
4 \mbox{ or } 5 & 0  & 1 \mbox{ or } 2 & 1 \mbox{ or } 2 & 3 & 4 \mbox{ or } 5 \\
\end{array}
\right).
$
The possible base permutations are then $\beta_1=\psi^5$ (which corresponds to the conjugacy class calculated above),
$\beta_2=(0,5,4,3,1)$, $\beta_3=(0,4,3,2,1)$, and $\beta_4=(0,4,3,1)$.
As we shall presently see, unlike the previous example, each candidate base permutation yields permutiples.

We now determine the conjugacy class for $\beta_2=(0,5,4,3,1)$.
In order to do this, we must first find an example from this class. 
We shall see that the techniques we used for the previous base-6 examples will not work in this case.
Therefore, we must appeal to more rudimentary techniques. In particular,
Theorem \ref{fund} will prove useful.
We need permutations $\tau_2$ and $\pi$ yielding a $(3,12,\tau_2)$-permutiple such that 
$\pi \tau_2 \pi^{-1}=(0,5,4,3,1)$. 
Rearranging, we have $\tau_2=(\pi^{-1}(0),\pi^{-1}(5),\pi^{-1}(4),\pi^{-1}(3),\pi^{-1}(1))$.
In order to reduce the number of candidate permutations, we attempt to find a value of $\pi$
which fixes 0. Then $\tau_2=(0,\pi^{-1}(5),\pi^{-1}(4),\pi^{-1}(3),\pi^{-1}(1))$.
Applying  Theorem \ref{n_div_b}, the new permutiple
$(d_{\pi(5)},d_{\pi(4)},d_{\pi(3)},d_{\pi(2)},d_{\pi(1)},d_{\pi(0)})_{12}$
has the carry vector $(c_{\pi(5)},c_{\pi(4)},c_{\pi(3)},c_{\pi(2)},c_{\pi(1)},c_{\pi(0)})$.
Applying Theorem \ref{fund} to the $j$th digit, we have that $3d_{\pi \tau_2(j)}-d_{\pi(j)}=12 c_{\pi \psi (j)}-c_{\pi(j)}$.
We know by the above that $\pi(0)=0$ and $\tau_2(0)=\pi^{-1}(5)$. Therefore, for $j=0$, we have
$3d_{5}-d_{0}=3\cdot 5 -3 =12 c_{\pi(1)}$. Thus, $c_{\pi(1)}=1=c_1 \mbox{ or } c_4$. 
So $\pi(1)$ must be either 1 or 4. We shall attempt to find a solution with $\pi(1)=1$.
Then $\tau_2=(0,\pi^{-1}(5),\pi^{-1}(4),\pi^{-1}(3),1)$.
Another application of Theorem \ref{fund} for $j=1$ yields
$3d_{\pi \tau_2(1)}-d_{\pi(1)}=12 c_{\pi \psi (1)}-c_{\pi(1)}$, or 
$-1=3 \cdot 3 - 10=3d_{0}-d_{1}=12 c_{\pi(2)}-c_{1}=12c_{\pi(2)}-1$. 
Thus, $c_{\pi(2)}=0$, so $\pi(2)$ equals either 0 or 2, but since $\pi(0)=0$, it must be that
$\pi(2)=2$. Applying Theorem \ref{fund} for $j=2$, we obtain
$3d_{\pi \tau_2(2)}-d_{\pi(2)}=12 c_{\pi \psi (2)}-c_{\pi(2)}$, which, using the above information
reduces to $c_{\pi(3)}=1$. That is, $\pi(3)$ is either 1 or 3. But since $\pi(1)=1$,
it follows that $\pi(3)=4$. Thus, $\tau_2=(0,\pi^{-1}(5),3,\pi^{-1}(3),1)$.
In the above fashion, Theorem \ref{fund} for $j=3$ then gives that 
$c_{\pi(4)}=2$. Then $\pi(4)$ must equal either 3 or 5. 
Letting $\pi(4)=3$, we have $\tau_2=(0,5,3,4,1)$. 

From the above, $\pi=(4,3)$ and  $\tau_2=(0,5,3,4,1)$ give us a solution
from the conjugacy class corresponding to $\beta_2=(0,5,4,3,1)$, namely,
the $(3,12,\tau_2)$-permutiple,
$(d_{\pi(5)},d_{\pi(4)},d_{\pi(3)},d_{\pi(2)},d_{\pi(1)},d_{\pi(0)})_{12}
=(5, 8, 1, 6, 10, 3)_{12} = 3 \cdot (1, 10, 8, 6, 3, 5)_{12}$, with carries
$(\hat{c}_5,\hat{c}_4,\hat{c}_3,\hat{c}_2,\hat{c}_1,\hat{c}_0)=(2, 2, 1, 0, 1, 0)$.
Taking $\hat{p}=(\hat{d}_{5},\hat{d}_{4},\hat{d}_{3},\hat{d}_{2},\hat{d}_{1},\hat{d}_{0})_{12}
=(5, 8, 1, 6, 10, 3)_{12}$ to be our reference example from this class, we now use Corollary \ref{cycle_theorem}
to find the remaining elements of this class:
 $$
 \left[
\begin{array}{c}
1\\
0\\
1\\
2\\
2\\
0\\
\end{array}
\right]
=
P_{\pi \psi \pi^{-1}}
\left[
\begin{array}{c}
0\\
1\\
0\\
1\\
2\\
2\\
\end{array}
\right].
 $$ 
We express the above in component form as
$$
\begin{array}{l}
\hat{c}_{\pi \psi \pi^{-1}(0)}=1=\hat{c}_1 \mbox{ or } \hat{c}_3,\\
\hat{c}_{\pi \psi \pi^{-1}(1)}=0=\hat{c}_0 \mbox{ or } \hat{c}_2,\\
\hat{c}_{\pi \psi \pi^{-1}(2)}=1=\hat{c}_1 \mbox{ or } \hat{c}_3,\\
\hat{c}_{\pi \psi \pi^{-1}(3)}=2=\hat{c}_4 \mbox{ or } \hat{c}_5,\\
\hat{c}_{\pi \psi \pi^{-1}(4)}=2=\hat{c}_4 \mbox{ or } \hat{c}_5,\\
\hat{c}_{\pi \psi \pi^{-1}(5)}=0=\hat{c}_0 \mbox{ or } \hat{c}_2.\\
\end{array}
$$ 
Therefore, $\pi \psi \pi^{-1}$
can be expressed as
$
\left(
\begin{array}{cccccc}
0 & 1 & 2 & 3 & 4 & 5\\
1 \mbox{ or } 3 & 0 \mbox{ or } 2  & 1 \mbox{ or } 3 &  5 & 4 & 0 \mbox{ or } 2 \\
\end{array}
\right).
$
In particular, we have either 
$(0,1,2,3,4,5)$, or $(0,3,4,5,2,1)$, and since $\hat{c}_0=\hat{c}_2=0$, 
the 4 permutations induced by the digits of $\hat{p}$ are 
 the identity $\varepsilon$ (corresponding to $\hat{p}$),
$
\pi_2=
\left(
\begin{array}{cccccc}
0 & 1 & 2 & 3 & 4 & 5\\
2 & 3 & 4 & 5 & 0 & 1 \\
\end{array}
\right),
$
$
\pi_3=
\left(
\begin{array}{cccccc}
0 & 1 & 2 & 3 & 4 & 5\\
0 & 3 & 4 & 5 & 2 & 1 \\
\end{array}
\right),
$
and
$
\pi_4=
\left(
\begin{array}{cccccc}
0 & 1 & 2 & 3 & 4 & 5\\
2 & 1 & 0 & 3 & 4 & 5 \\
\end{array}
\right).
$
The conjugacy class corresponding to $\beta_2=(0,5,4,3,1)$ is given by the table below. 
\begin{center}
 \begin{tabular}{|c|c|c|c|}
\hline
$(\hat{d}_{\pi(5)},\hat{d}_{\pi(4)},\hat{d}_{\pi(3)},\hat{d}_{\pi(2)},\hat{d}_{\pi(1)},\hat{d}_{\pi(0)})_{12}$ & $\pi$ & $\tau$ & $(\hat{c}_{\pi(5)},\hat{c}_{\pi(4)},\hat{c}_{\pi(3)},\hat{c}_{\pi(2)},\hat{c}_{\pi(1)},\hat{c}_{\pi(0)})$\\\hline
$(5,8,1,6,10,3)_{12}$ & $\varepsilon$ & $\tau_2=(0,5,3,4,1)$ & $(2,2,1,0,1,0)$\\\hline
$(10,3,5,8,1,6)_{12}$ & $\pi_2$ & $\pi_2^{-1} \tau_2 \pi_2$  & $(1,0,2,2,1,0)$\\\hline
$(10,6,5,8,1,3)_{12}$ & $\pi_3$ & $\pi_3^{-1} \tau_2 \pi_3$ &$(1,0,2,2,1,0)$\\\hline 
$(5,8,1,3,10,6)_{12}$ & $\pi_4$ & $\pi_3^{-1} \tau_2 \pi_4$ & $(2,2,1,0,1,0)$\\\hline 
\end{tabular}
\end{center}

Finding an example from the conjugacy class corresponding to $\beta_3=(0,4,3,2,1)$,
repeated use of Theorem \ref{fund} as above yields the permutations $\pi=(1,2,3,4,5)$ and
$\tau_3=(5,0,3,2,1)$. From these we obtain the $(3,12,\tau_3)$-permutiple, 
$(d_{\pi(5)},d_{\pi(4)},d_{\pi(3)},d_{\pi(2)},d_{\pi(1)},d_{\pi(0)})_{12}
=(10, 5, 1, 8, 6, 3)_{12} = 3 \cdot (3, 5, 8, 6, 10, 1)_{12}$, with carries 
$(\hat{c}_5,\hat{c}_4,\hat{c}_3,\hat{c}_2,\hat{c}_1,\hat{c}_0)=(1, 2, 1, 2, 0, 0)$.

In similar fashion to calculations for the above conjugacy classes, the reader may verify that this example 
via Corollary \ref{cycle_theorem} gives the following four permutations:
 the identity $\varepsilon$,
$
\pi_5=
\left(
\begin{array}{cccccc}
0 & 1 & 2 & 3 & 4 & 5\\
1 & 2 & 3 & 4 & 5 & 0 \\
\end{array}
\right),
$
$
\pi_6=
\left(
\begin{array}{cccccc}
0 & 1 & 2 & 3 & 4 & 5\\
0 & 1 & 4 & 3 & 2 & 5 \\
\end{array}
\right),
$
and
$
\pi_7=
\left(
\begin{array}{cccccc}
0 & 1 & 2 & 3 & 4 & 5\\
1 & 4 & 3 & 2 & 5 & 0 \\
\end{array}
\right),
$
and thus, the conjugacy class corresponding to $\beta_3=(0,4,3,2,1)$ is given by the following table. 
\begin{center}
 \begin{tabular}{|c|c|c|c|}
\hline
$(\hat{d}_{\pi(5)},\hat{d}_{\pi(4)},\hat{d}_{\pi(3)},\hat{d}_{\pi(2)},\hat{d}_{\pi(1)},\hat{d}_{\pi(0)})_{12}$ & $\pi$ & $\tau$ & $(\hat{c}_{\pi(5)},\hat{c}_{\pi(4)},\hat{c}_{\pi(3)},\hat{c}_{\pi(2)},\hat{c}_{\pi(1)},\hat{c}_{\pi(0)})$\\\hline
$(10, 5, 1, 8, 6, 3)_{12}$ & $\varepsilon$ & $\tau_3=(5,0,3,2,1)$ & $(1, 2, 1, 2, 0, 0)$\\\hline
$(3, 10, 5, 1, 8, 6)_{12}$ & $\pi_5$ & $\pi_5^{-1} \tau_3 \pi_5$ & $(0, 1, 2, 1, 2, 0)$\\\hline 
$(10, 8, 1, 5, 6, 3)_{12}$ & $\pi_6$ & $\pi_6^{-1} \tau_3 \pi_6$ &$(1, 2, 1, 2, 0, 0)$\\\hline 
$(3, 10, 8, 1, 5, 6)_{12}$ & $\pi_7$ & $\pi_7^{-1} \tau_3 \pi_7$ &$(0, 1, 2, 1, 2, 0)$\\\hline
\end{tabular}
\end{center}

To an example from our final conjugacy class corresponding to $\beta_4=(0,4,3,1)$,
we use Theorem \ref{fund} as above to find the permutations $\pi=(5,1,2,4)$
and $\tau_4=(0,2,3,5)$.
From these we obtain the $(3,12,\tau_4)$-permutiple, 
$(d_{\pi(5)},d_{\pi(4)},d_{\pi(3)},d_{\pi(2)},d_{\pi(1)},d_{\pi(0)})_{12}
=(10, 5, 8, 1, 6, 3)_{12} = 3 \cdot (3, 5, 10, 8, 6, 1)_{12}$, with carries 
$(\hat{c}_5,\hat{c}_4,\hat{c}_3,\hat{c}_2,\hat{c}_1,\hat{c}_0)=(1, 2, 2, 1, 0, 0)$.

Again, we shall leave the details of the calculations involving Corollary \ref{cycle_theorem} to
the reader who may verify that this permutiple yields the following two permutations: the identity $\varepsilon$ and 
$
\pi_8=
\left(
\begin{array}{cccccc}
0 & 1 & 2 & 3 & 4 & 5\\
1 & 2 & 3 & 4 & 5 & 0\\
\end{array}
\right).
$
The conjugacy class for $\beta_4=(0,4,3,1)$ is given by the table below. 
\begin{center}
 \begin{tabular}{|c|c|c|c|}
\hline
$(\hat{d}_{\pi(5)},\hat{d}_{\pi(4)},\hat{d}_{\pi(3)},\hat{d}_{\pi(2)},\hat{d}_{\pi(1)},\hat{d}_{\pi(0)})_{12}$ & $\pi$ & $\tau$ & $(\hat{c}_{\pi(5)},\hat{c}_{\pi(4)},\hat{c}_{\pi(3)},\hat{c}_{\pi(2)},\hat{c}_{\pi(1)},\hat{c}_{\pi(0)})$\\\hline
$(10, 5, 8, 1, 6, 3)_{12}$ & $\varepsilon$ & $\tau_4=(0,2,3,5)$  &$(1, 2, 2, 1, 0, 0)$\\\hline 
$(3, 10, 5, 8, 1, 6)_{12}$ & $\pi_8$ & $\pi_8^{-1} \tau_4 \pi_8$ & $(0, 1, 2, 2, 1, 0)$\\\hline
\end{tabular}
\end{center}

\end{example}

\section{Future Directions and Concluding Remarks}

While we have developed methods for finding all permutiples having the same digits and carries
as a known example (which, as we have seen, allow us to find all desired examples when $n$ divides $b$), 
the next step forward is to find methods which give us all the desired examples. 

Aside from the above problem, the work we have done here leaves many questions.
Are there other conditions or divisibility criteria which, using the methods developed here or some
slight variation thereof, allow us to find all permutiples having the same digits as a
single permutiple example already in hand?
How do the orders of 
base permutations of each conjugacy class relate to the order of $\sigma$?
Are there any restrictions on the size of the conjugacy classes?  
Regarding the general permutiple problem, 
other questions certainly abound: what kinds of permutations can $\sigma$ be?
What sort of restrictions might there be on the order of $\sigma$? 

As we have already seen, repeated application of Theorem \ref{fund}, Corollary \ref{cycle_theorem},
and Corollary \ref{base_perm} give us 
methods for finding permutations of the digits of a starting permutiple which yield other permutiples.
However, in fairness, we should add that some examples require substantially more calculation than others in order 
to find every example. 
For instance, considering the base-10 cyclic example given in the introductory paragraph, $714285 = 5 \cdot 142857$,
the reader may verify that Corollary \ref{base_perm} requires us to examine 36 possible base permutations,
among which, only four actually yield non-empty conjugacy classes. 
This is all to say that future efforts should look for criteria for determining when base permutations
yield non-empty conjugacy classes.

Another avenue of investigation is finding a way to generalize Sloane's \cite{sloane} Young graph representation
of palintiple structure in order to visualize permutiple structure. 
We imagine that such a construction would allow us to classify, and better understand,
permutiple structure.
We also imagine that it would be substantially more complex.
Finally, as Young graphs themselves are an interesting area of study in their own right,
we conjecture that its generalization would also justify future study.

It is also worth mentioning that the matrices $nP_{\sigma}-I$ and $bP_{\psi}-I$, as well as their inverses 
(given by
$\frac{1}{n^{|\sigma|}-1} \sum_{\ell=0}^{|\sigma|-1}(nP_{\sigma})^{\ell}$ and
$\frac{1}{b^{k+1}-1} \sum_{\ell=0}^{k}(bP_{\psi})^{\ell}$, respectively), all have interesting properties. 
For instance, both $bP_{\psi}-I$ and its inverse are circulant matrices.
Moreover, both $nP_{\sigma}-I$ and $bP_{\psi}-I$ have the property that
every column and row sum to $n-1$ and $b-1$, respectively.
In addition to these properties, we ask if these matrices are endowed with other special properties
when $(n,b,\sigma)$-permutiples exist.

With the palintiple problem firmly in mind,
the difficulty of particular digit permutation problems might make the general permutiple problem seem intractable.
However, the methods developed here seem to prove otherwise; there is certainly structure that one can take 
advantage of. Moreover, as noted by Holt \cite{holt_3}, studying the general problem may very well 
offer insight into particular problems which study only one kind of permutation. 
In particular, it might be possible to derive (in the manner described by Holt \cite{holt_3}) entire palintiple classes
from certain permutiple types such as those mentioned by Holt \cite{holt_3}. 
What is more, it seems that general permutiples can also be derived from other permutiples.
As a noteworthy example, we present the cyclic $(6,12,\psi^3)$-permutiple, $(10,3,5,1,8,6)_{12}=6\cdot(1,8,6,10,3,5)_{12}$, 
whose non-zero carries are the digits of the $(2,6)$-palintiple which served as our initial base-6 example, 
$(4,3,5,1,2)_6=2 \cdot (2,1,5,3,4)_6$. 
Another example is the starting permutiple of our last example, 
$(5, 1, 8, 6, 10, 3)_{12} = 3 \cdot (1, 8, 6, 10, 3, 5)_{12}$, with carries $(c_5,c_4,c_3,c_2,c_1,c_0)=(2, 1, 2, 0, 1, 0)$.
The carries (excluding $c_0=0$) are the digits of a $(2,3)$-palintiple, which again, as shown by Holt \cite{holt_3}, gives
rise to an entire family of derived palintiples.
When choosing our examples, we did not intentionally seek these out, 
but rather, we only later noticed that such examples of ``derived permutiples'' seem to naturally abound. 
We suspect that there are deep and interesting connections between these examples which have yet to be explored.


\end{document}